\documentclass[12pt, reqno]{amsart}

\usepackage{amsmath, amsthm, amscd, amsfonts, amssymb, graphicx, color}
\usepackage[bookmarksnumbered, colorlinks, plainpages]{hyperref}
\usepackage{mathtools}
\usepackage{enumerate}
\usepackage{amscd} 
\usepackage{xy}
\xyoption{all}

\newtheorem{theorem}{Theorem}[section]
\newtheorem{lemma}[theorem]{Lemma}

\newtheorem{corollary}[theorem]{Corollary}
\theoremstyle{definition}
\newtheorem{definition}[theorem]{Definition}

\newtheorem{remark}[theorem]{Remark}

\begin{document}
\setcounter{page}{1}

\title[ On Some Properties of the Trigamma Function ]{On Some Properties of the Trigamma Function }

\author[K. Nantomah, G. Abe-I-Kpeng and S. Sandow]{Kwara Nantomah$^{1,*}$, Gregory Abe-I-Kpeng$^2$ and Sunday Sandow$^3$}

\address{$^1$Department of Mathematics, School of Mathematical Sciences, C. K. Tedam University of Technology and Applied Sciences, P. O. Box 24, Navrongo, Upper-East Region, Ghana. }
\email{\textcolor[rgb]{0.00,0.00,0.84}{ knantomah@cktutas.edu.gh}}

\address{$^2$Department of Industrial Mathematics, School of Mathematical Sciences, C. K. Tedam University of Technology and Applied Sciences, P. O. Box 24, Navrongo, Upper-East Region, Ghana. }
\email{\textcolor[rgb]{0.00,0.00,0.84}{ gabeikpeng@cktutas.edu.gh}}

\address{$^3$Department of Mathematics, School of Mathematical Sciences, C. K. Tedam University of Technology and Applied Sciences, P. O. Box 24, Navrongo, Upper-East Region, Ghana. }
\email{\textcolor[rgb]{0.00,0.00,0.84}{ ssandow@cktutas.edu.gh}}




\date{Received: xxxxxx; Revised: xxxxxx; Accepted: xxxxxx.
\newline \indent $^{*}$ Corresponding author}

\begin{abstract}
In 1974, Gautschi proved an intriguing inequality involving the gamma function $\Gamma$. Precisely, he proved that, for $z>0$, the harmonic mean of $\Gamma(z)$ and $\Gamma(1/z)$ can never be less than 1. In 2017, Alzer and Jameson extended this result to the digamma function $\psi$ by proving that,  for $z>0$, the harmonic mean of $\psi(z)$ and $\psi(1/z)$ can never be less than $-\gamma$ where $\gamma$ is the Euler-Mascheroni constant. In this paper, our goal is to extend the results to the trigamma function $\psi'$. We prove among other things that, for $z>0$, the harmonic mean of $\psi'(z)$ and $\psi'(1/z)$ can never be greater than $\pi^2/6$.
\end{abstract} \maketitle


\section{Introduction}

The classical gamma function which is an extension of the factorial function is frequently defined  as 
\begin{equation*} \label{eqn:Digamma-Integral-Rep}
\Gamma(z)=\int_{0}^{\infty}r^{z-1}e^{-r}dr
\end{equation*}
for $z>0$.
Closely connected to the gamma function is the the digamma (or psi) function is which is defined as
\begin{align}
\psi(z)=\frac{d}{dz}\ln \Gamma(z)&= -\gamma + \int_{0}^{\infty} \frac{e^{-r} - e^{-zr}}{1-e^{-r}}dr,   \\
&=  \int_{0}^{\infty} \left( \frac{e^{-r}}{r} - \frac{ e^{-zr}}{1-e^{-r}} \right)dr,  \\
&= -\gamma - \frac{1}{z} + \sum_{n=1}^{\infty}\frac{z}{n(n+z)},
\end{align}
where $\gamma$ is the Euler-Mascheroni constant. Derivatives of the digamma function which are called polygamma functions are defined as 
\begin{align}
\psi^{(c)}(z)
&=  (-1)^{c+1}\int_{0}^{\infty} \frac{r^{c}e^{-zr}}{1-e^{-r}}dr,  \label{eqn:Polygamma-Integral-Rep} \\
& = (-1)^{c+1} \sum_{n=0}^{\infty}\frac{c!}{(n+z)^{c+1}},  \label{eqn:Polygamma-Series-Rep}
\end{align}
for $z>0$ and $c\in \mathbb{N}$. The particular case $\psi'(z)$ is what is referred to as the trigamma function. Also, it is well known in the literature that the integral
\begin{equation}\label{eqn:integral-rep-vip}
\frac{c!}{z^{c+1}} = \int_{0}^{\infty}r^{c}e^{-zr}\,dr
\end{equation}
holds for $z>0$ and $c\in \mathbb{N}_0$.

In 1974, Gautschi \cite{Gautschi-1974-SJMA} presented an elegant inequality involving the gamma function. Precisely, he proved that, for $z>0$, the harmonic mean of $\Gamma(z)$ and $\Gamma(1/z)$ is at least 1. That is,
\begin{equation}\label{eqn:Gautschi-Ineq-Gamma}
\frac{2\Gamma(z)\Gamma(1/z)}{\Gamma(z)+\Gamma(1/z)}\geq1,  
\end{equation}
for $z>0$ and with equality when $z=1$. As a direct consequence of \eqref{eqn:Gautschi-Ineq-Gamma}, the inequalities
\begin{equation}\label{eqn:Gautschi-Ineq-Gamma-Impl-1}
\Gamma(z) + \Gamma(1/z) \geq 2
\end{equation}
and 
\begin{equation}\label{eqn:Gautschi-Ineq-Gamma-Impl-2}
\Gamma(z)\Gamma(1/z) \geq 1
\end{equation} 
are obtained for $z>0$. Attributing to the importance of this  inequality, some refinements and extensions have been investigated \cite{Alzer-1997-JCAM, Alzer-1999-PAMS, Alzer-2002-PEMS, Alzer-2003-JCAM, Alzer-2006-JCAM, Alzer-2008-NA, Giordano-Laforgia-2001-JCAM, Jameson-Jameson-2012-JMI}.

In 2017, Alzer and Jameson \cite{Alzer-Jameson-2017-RSMUP} established a striking companion of \eqref{eqn:Gautschi-Ineq-Gamma} which involves the digamma function $\psi(z)$. 
They established that the inequality 
\begin{equation}\label{eqn:Gautschi-Ineq-Digamma}
\frac{2\psi(z)\psi(1/z)}{\psi(z)+\psi(1/z)} \geq -\gamma
\end{equation}
holds for $z>0$ and with equality when $z=1$. Thereafter, Alzer \cite{Alzer-2017-RSMUPT}  refined \eqref{eqn:Gautschi-Ineq-Digamma} by proving that
\begin{equation}\label{eqn:Gautschi-Ineq-Digamma-Refinement}
\frac{2\psi(z)\psi(1/z)}{\psi(z)+\psi(1/z)} \geq -\gamma \frac{2z}{z^2+1} 
\end{equation}
holds for $z>0$ and with equality when $z=1$.

In 2018, Yin et al. \cite{Yin-Etal-2018-ADE} extended inequality \eqref{eqn:Gautschi-Ineq-Digamma} to the $k$-analogue of the digamma function by establishing that
\begin{equation}\label{eqn:Gautschi-Ineq-k-Digamma-Yin}
\frac{2\psi_{k}(z)\psi_{k}(1/z)}{\psi_{k}(z)+\psi_{k}(1/z)} \geq \frac{\ln^2k+\gamma^2-2(\gamma+1)\ln k}{k\left[\ln k+\psi(1/k) \right]}
\end{equation}
for $z>0$ and $\frac{1}{\sqrt[3]{3}}\leq k \leq 1$.

In 2020, Yildirim \cite{Yildirim-2020-JMI} improved on the inequality \eqref{eqn:Gautschi-Ineq-k-Digamma-Yin} by establishing that
\begin{equation}\label{eqn:Gautschi-Ineq-k-Digamma-Yildirim}
\frac{2\psi_{k}(z)\psi_{k}(1/z)}{\psi_{k}(z)+\psi_{k}(1/z)} \geq \psi_{k}(1)
\end{equation}
for $z>0$ and $k>0$. When $k=1$, inequalities \eqref{eqn:Gautschi-Ineq-k-Digamma-Yin}  and \eqref{eqn:Gautschi-Ineq-k-Digamma-Yildirim} both return to inequality \eqref{eqn:Gautschi-Ineq-Digamma}.

In 2021, Bouali \cite{Bouali-2021-FILOMAT} extended inequalities \eqref{eqn:Gautschi-Ineq-Gamma} and \eqref{eqn:Gautschi-Ineq-Digamma} to the $q$-analogues of the gamma and digamma functions by proving that
\begin{equation}\label{eqn:HMI-q-Gamma}
\frac{2\Gamma_{q}(z)\Gamma_{q}(1/z)}{\Gamma_{q}(z)+\Gamma_{q}(1/z)} \geq1
\end{equation}
for $z>0$ and
\begin{equation}\label{eqn:HMI-q-Digamma}
\frac{2\psi_{q}(z)\psi_{q}(1/z)}{\psi_{q}(z)+\psi_{q}(1/z)} \geq \psi_{q}(1)
\end{equation}
for $z>0$ and $q\in(0,p_0)$, where $p_0\simeq 3.239945$. 

For similar results involving other special functions, one may refer to the works \cite{Matejicka-2019-PAIS, Nantomah-2019-BIMVI, Nantomah-2020-EJMS, Nantomah-2021-IJAM, Nantomah-2021-AMSJ, Nantomah-2023-AJMS}.

In the present investigation, our goal is to extend the results of Alzer and Jameson \cite{Alzer-Jameson-2017-RSMUP} to the trigamma function $\psi'$. Specifically, we prove among other things that, for $z>0$, the harmonic mean of $\psi'(z)$ and $\psi'(1/z)$ can never be greater than $\pi^2/6$. We present our results in Section \ref{sec:Sec-Two}. In order to establish our results, we require the following preliminary definitions and lemmas.

\begin{definition}[\cite{Niculescu-2000-MIA}]\label{def:GG-Convex-Funct}
A function $H:\mathcal{I}\subseteq \mathbb{R}^+ \rightarrow \mathbb{R}$ is referred to as GG-convex if 
\begin{equation}\label{eqn:GG-Convex-Funct}
H(x^{1-k}y^{k}) \leq H(x)^{1-k} H(y)^{k}
\end{equation}
for all $x,y\in \mathcal{I}$ and $k\in[0,1]$. If the inequality in \eqref{eqn:GG-Convex-Funct} is reversed, then $H$ is said to be GG-concave.
\end{definition}

\begin{definition}[\cite{Niculescu-2000-MIA}]\label{def:GA-Convex-Funct}
A function $H:\mathcal{I}\subseteq \mathbb{R}^+ \rightarrow \mathbb{R}$ is referred to as GA-convex if 
\begin{equation}\label{eqn:GA-Convex-Funct}
H(x^{1-k}y^{k}) \leq (1-k)H(x) + kH(y)
\end{equation}
for all $x,y\in \mathcal{I}$ and $k\in[0,1]$. If the inequality in \eqref{eqn:GA-Convex-Funct} is reversed, then $H$ is said to be GA-concave.
\end{definition}

\begin{lemma}[\cite{Niculescu-2000-MIA}]\label{lem:Condition-GG-Convex}
A function $H:\mathcal{I}\subseteq \mathbb{R}^+ \rightarrow \mathbb{R}$ is GG-convex (or GG-concave) if and only if $\frac{zH'(z)}{H(z)}$ is increasing (or decreasing) on $\mathcal{I}$ respectively.
\end{lemma}

\begin{lemma}[\cite{Zhang-Chu-Zhang-2010-JIA}]\label{lem:Condition-GA-Convex}
A function $H:\mathcal{I}\subseteq \mathbb{R}^+ \rightarrow \mathbb{R}$ is GA-convex if and only if
\begin{equation}\label{eqn:Condition-GA-Convex}
H'(z) + z H''(z) \geq 0
\end{equation}
for all $z\in \mathcal{I}$. The function $H$ is said to be GA-concave if and only if  the inequality in \eqref{eqn:Condition-GA-Convex} is reversed.
\end{lemma}

The following lemma is well known in the literature as the convolution theorem for Laplace transforms.

\begin{lemma}\label{lem:Lap-Tran-of-Conv}
Let $f(r)$ and $g(r)$ be any two functions with convolution $f\ast g= \int_{0}^{r}f(r-s)g(s)\,ds$. Then the Laplace transform of the convolution is given as
\begin{equation*}
 \mathcal{L}\left \{ f\ast g\right \}= \mathcal{L}\left \{ f\right \} \mathcal{L}\left \{g\right \} .
\end{equation*}
In other words,
\begin{equation}\label{eqn:Lap-Tran-of-Conv}
\int_{0}^{\infty}\left[\int_{0}^{r}f(r-s)g(s)\,ds \right]e^{-zr}\,dr =
\int_{0}^{\infty}f(r)e^{-zr}\,dr \int_{0}^{\infty}g(r)e^{-zr}\,dr .
\end{equation}
\end{lemma}

\begin{lemma}[\cite{Pinelis-2002-JIPAM}]\label{lem:LMR}
Let $-\infty \leq u<v \leq \infty$ and $p$ and $q$ be continuous functions that are differentiable on $(u,v)$, with $p(u+)=q(u+)=0$ or $p(v-)=q(v-)=0$. Suppose that $q(z)$ and $q'(z)$ are nonzero for all $z\in(u,v)$. If $\frac{p'(z)}{q'(z)}$ is increasing (or decreasing) on $(u,v)$, then $\frac{p(x)}{q(x)}$ is also  increasing (or decreasing) on $(u,v)$.
\end{lemma}

\section{Results}\label{sec:Sec-Two}

\begin{theorem} \label{lem:GG-Convex-Trigamma}
The function $\psi'(z)$ is GG-convex on $(0,\infty)$. In other words,
\begin{equation}\label{eqn:GG-Convex-Trigamma}
\psi'(x^{1-k}y^{k}) \leq \left[\psi'(x)\right]^{1-k}\left[\psi'(y)\right]^{k}
\end{equation}
is satisfied for $x>0$, $y>0$ and $k\in[0,1]$.
\end{theorem}

\begin{proof}
As a result of Lemma \ref{lem:Condition-GG-Convex}, it suffices to show that the function $z\frac{\psi''(z)}{\psi'(z)}$ is increasing on $(0,\infty)$ and this follows from Lemma 2 of \cite{Alzer-2017-RSMUPT}.
\end{proof}

\begin{corollary} \label{lem:Product-Ineq-Trigamma}
The inequality 
\begin{equation}\label{eqn:Product-Ineq-Trigamma}
\psi'(z) \psi'(1/z)  \geq \left(\frac{\pi^2}{6}\right)^2
\end{equation}
holds for $z\in(0,\infty)$ and with equality when $z=1$.
\end{corollary}

\begin{proof}
By letting $x=z$, $y=1/z$ and $k=\frac{1}{2}$ in \eqref{eqn:GG-Convex-Trigamma}, we obtain
\begin{equation*}
\sqrt{\psi'(z) \psi'(1/z)}  \geq \psi'(1)=\frac{\pi^2}{6}
\end{equation*}
which gives the desired result.
\end{proof}

\begin{lemma}\label{lem:Special-Ineq}
For $r>0$, we have
\begin{equation}\label{eqn:Special-Ineq}
0 < \frac{re^{-r}}{1-e^{-r}} < 1.
\end{equation}
\end{lemma}

\begin{proof}
By direct computation, we obtain 
\begin{equation*}
\mathcal{B}(r) =\frac{re^{-r}}{1-e^{-r}} = \frac{p_1(r)}{q_1(r)} 
\end{equation*}
where $p_1(r)=re^{-r}$, $q_1(r)=1-e^{-r}$ and $p_1(0+)=q_1(0+)=0$. Then
\begin{equation*}
\frac{p'_1(r)}{q'_1(r)}=1-r
\end{equation*}
and then
\begin{equation*}
\left(\frac{p'_1(r)}{q'_1(r)}\right)'=-1<0.
\end{equation*}
Thus, $\frac{p'_1(r)}{q'_1(r)}$ is decreasing and as a result of Lemma \ref{lem:LMR}, the function $\mathcal{B}(r)$ is also decreasing. Hence
\begin{equation*}
0=\lim_{r\to \infty} \mathcal{B}(r)<\mathcal{B}(r)<\lim_{r\to0+} \mathcal{B}(r)=1
\end{equation*}
which completes the proof.
\end{proof}

\begin{theorem} \label{lem:GA-Convex-Trigamma}
The function $\psi'(z)$ is GA-convex on $(0,\infty)$. In other words,
\begin{equation}\label{eqn:GA-Convex-Trigamma}
\psi'(x^{1-k}y^{k}) \leq (1-k)\psi'(x) + k \psi'(y)
\end{equation}
is satisfied for $x>0$, $y>0$ and $k\in[0,1]$.
\end{theorem}

\begin{proof}
As a result of Lemma \ref{lem:Condition-GA-Convex}, it suffices to show that 
\begin{equation}\label{eqn:Spec-Ineq-A}
\phi(z)=\psi''(z) + z \psi'''(z) \geq 0
\end{equation}
for $z\in(0,\infty)$.  By applying \eqref{eqn:Polygamma-Integral-Rep}, \eqref{eqn:integral-rep-vip} and Lemma \ref{lem:Lap-Tran-of-Conv}, we obtain
\begin{align*}
\frac{\phi(z)}{z}&=\frac{1}{z}\psi''(z) +  \psi'''(z) \\
&= - \int_{0}^{\infty} e^{-zr} dr  \int_{0}^{\infty} \frac{r^{2}e^{-zr}}{1-e^{-r}}dr + \int_{0}^{\infty} \frac{r^{3}e^{-zr}}{1-e^{-r}}dr  \\
&= - \int_{0}^{\infty}\left[ \int_{0}^{r} \frac{s^{2}}{1-e^{-s}}ds \right] e^{-zr} dr   + \int_{0}^{\infty} \frac{r^{3}e^{-zr}}{1-e^{-r}}dr  \\
&=\int_{0}^{\infty} \mathcal{A}(r) e^{-zr} dr 
\end{align*}
where
\begin{equation*}
\mathcal{A}(r)=\frac{r^3}{1-e^{-r}} - \int_{0}^{r} \frac{s^{2}}{1-e^{-s}}ds.
\end{equation*}
Then by direct computations and as  a result of \eqref{eqn:Special-Ineq}, we have
\begin{align*}
\mathcal{A}'(r)&=\frac{3r^2}{1-e^{-r}} - \frac{r^3e^{-r}}{(1-e^{-r})^2} - \frac{r^2}{1-e^{-r}} \\
&=\frac{r^2}{1-e^{-r}}\left[ 2- \frac{re^{-r}}{1-e^{-r}}\right] \geq 0.
\end{align*}
Hence $\mathcal{A}(r)$ is increasing and this implies that
\begin{equation*}
\mathcal{A}(r) \geq \lim_{r\to 0+}\mathcal{A}(r)=0.
\end{equation*}
Therefore, $\phi(z)\geq0$ which completes the proof.
\end{proof}

\begin{remark}
Inequality \eqref{eqn:Spec-Ineq-A} implies that the function $z\psi''(z)$ is increasing.
\end{remark}

\begin{corollary} \label{lem:Sum-Ineq-Trigamma}
The inequality 
\begin{equation}\label{eqn:Sum-Ineq-Trigamma}
\psi'(z) + \psi'(1/z)  \geq \frac{\pi^2}{3}
\end{equation}
holds for $z\in(0,\infty)$ and with equality when $z=1$.
\end{corollary}

\begin{proof}
By letting $x=z$, $y=1/z$ and $k=\frac{1}{2}$ in \eqref{eqn:GA-Convex-Trigamma}, we obtain
\begin{equation*}
\frac{\psi'(z)}{2} + \frac{\psi'(1/z)}{2}  \geq \psi'(1)=\frac{\pi^2}{6}
\end{equation*}
which gives the desired result.
\end{proof}


\begin{lemma}[\cite{Alzer-Wells-1998-SIAMJMA}]\label{lem:Special-Ineq-2}
For $z>0$, the inequality
\begin{equation}\label{eqn:Special-Ineq-2}
\psi'(z)\psi'''(z)-2\left[ \psi''(z) \right]^2 \leq0
\end{equation}
is satisfied.
\end{lemma}

\begin{lemma}\label{lem:Decreasing-Property}
For $z>0$, the function
\begin{equation}\label{eqn:Decreasing-Property}
F(z)=\frac{z \psi''(z)}{[\psi'(z)]^2} 
\end{equation}
is decreasing.
\end{lemma}

\begin{proof}
By applying Lemma \ref{lem:Special-Ineq-2}, we obtain
\begin{align*}
[\psi'(z)]^3F'(z)&=\psi'(z)\psi''(z) + z\psi'(z)\psi'''(z) -2z[\psi''(z)]^2\\
&=\psi'(z)\psi''(z) + z \left[ \psi'(z)\psi'''(z)-2\left[ \psi''(z) \right]^2\right]\\
&<0.
\end{align*}
Hence $F'(z)<0$ which completes the proof.
\end{proof}

\begin{theorem}\label{thm:HMI-Shi}
For $z>0$, the inequality
\begin{equation}\label{eqn:HMI-Trigamma}
\frac{2\psi'(z) \psi'(1/z)}{\psi'(z) + \psi'(1/z)} \leq \frac{\pi^2}{6}
\end{equation}
holds and equality is attained if $z=1$.
\end{theorem}

\begin{proof}
The case for $z=1$ is apparent. For this reason, we only prove the case for $z\in(0,1)\cup(1,\infty)$.  Let 
\begin{equation*}
\mathcal{K}(z)=\frac{2\psi'(z) \psi'(1/z)}{\psi'(z) + \psi'(1/z)} \quad \text{and} \quad \beta(z)=\ln \mathcal{K}(z)
\end{equation*}
for $z\in(0,1)\cup(1,\infty)$. Then direct calculations gives
\begin{equation*}
\beta'(z)=\frac{\psi''(z)}{\psi'(z)} - \frac{1}{z^2}\frac{\psi''(1/z)}{\psi'(1/z)} - \frac{\psi''(z)-\frac{1}{z^2}\psi''(1/z)}{\psi'(z) + \psi'(1/z)}
\end{equation*}
which implies that
\begin{equation*}
z\left[ \psi'(z) + \psi'(1/z) \right]\beta'(z)=z\frac{\psi''(z)}{\psi'(z)}\psi'(1/z)  - \frac{1}{z}\frac{\psi''(1/z)}{\psi'(1/z)} \psi'(z).
\end{equation*}
This further gives rise to
\begin{align*}
z\left[ \frac{1}{\psi'(z) }+ \frac{1}{\psi'(1/z)} \right]\beta'(z)
&=z\frac{\psi''(z)}{[\psi'(z)]^2} - \frac{1}{z}\frac{\psi''(1/z)}{[\psi'(1/z)]^2} \\
&:=T(z).
\end{align*}
As a result of Lemma \ref{lem:Decreasing-Property}, we conclude that $T(z)>0$ if $z\in(0,1)$ and $T(z)<0$ if $z\in(1,\infty)$. Thus, $\beta(z)$ is increasing on $(0,1)$ and decreasing on $(1,\infty)$. Accordingly, $\mathcal{K}(z)$ is increasing on $(0,1)$ and decreasing on $(1,\infty)$. Therefore, on both intervals, we have
\begin{equation*}
\mathcal{K}(z)<\lim_{z\to1}\mathcal{K}(z)=\psi'(1)=\frac{\pi^2}{6}
\end{equation*}
completing the proof.
\end{proof}

\bibliographystyle{plain}

\begin{thebibliography}{99}

\bibitem{Alzer-1997-JCAM}  H. Alzer, \textit{A harmonic mean inequality for the gamma function}, J. Comp. Appl. Math., 87 (1997), 195-198.

\bibitem{Alzer-1999-PAMS}  H. Alzer, \textit{Inequalities for the gamma function}, Proc. Amer. Math. Soc., 128 (1999),
141-147.

\bibitem{Alzer-2002-PEMS}  H. Alzer, \textit{On a gamma function inequality of Gautschi}, Proc. Edinburgh Math. Soc., 45 (2002), 589-600.

\bibitem{Alzer-2003-JCAM}  H. Alzer, \textit{On Gautschi's harmonic mean inequality for the gamma function}, J. Comp. Appl. Math., 157 (2003), 243-249.

\bibitem{Alzer-2006-JCAM}  H. Alzer, \textit{Inequalities involving $\Gamma(x)$ and $\Gamma(1/x)$}, J. Comp. Appl. Math., 192 (2006), 460-480.

\bibitem{Alzer-2008-NA}  H. Alzer, \textit{Gamma function inequalities}, Numer. Algor., 49 (2008), 53-84.

\bibitem{Alzer-2017-RSMUPT} H. Alzer, \textit{A Mean Value Inequality for the Digamma Function}, Rendiconti Sem. Mat. Univ. Pol. Torino, 75(2)(2017), 19-25.

\bibitem{Alzer-Jameson-2017-RSMUP} H. Alzer and G. Jameson, \textit{A harmonic mean inequality for the digamma function and related results}, Rend. Sem. Mat. Univ. Padova., 137 (2017), 203-209.

\bibitem{Alzer-Wells-1998-SIAMJMA} H. Alzer and J. Wells, \textit{Inequalities for the Polygamma Functions}, SIAM J. Math. Anal., 29(6)(1998), 1459-1466.

\bibitem{Bouali-2021-FILOMAT}  M. Bouali, \textit{A harmonic mean inequality for the $q$-gamma and $q$-digamma functions}, Filomat 35(12)(2021), 4105-4119.

\bibitem{Gautschi-1974-SJMA} W. Gautschi, \textit{A harmonic mean inequality for the gamma function}, SIAM J. Math.
Anal., 5(1974), 278-281.

\bibitem{Giordano-Laforgia-2001-JCAM}  C. Giordano and A. Laforgia, \textit{Inequalities and monotonicity properties for the gamma function}, J. Comp. Appl. Math., 133 (2001), 387-396.

\bibitem{Jameson-Jameson-2012-JMI}  G. J. O. Jameson and  T. P. Jameson, \textit{An inequality for the gamma function conjectured by D. Kershaw}, J. Math. Ineq., 6 (2012), 175-181.

\bibitem{Matejicka-2019-PAIS}  L. Matejicka, \textit{Proof of a Conjecture On Nielsen's $\beta$-Function}, Probl. Anal. Issues Anal.,  8(26), (2019), 105-111.

\bibitem{Nantomah-2019-BIMVI}  K. Nantomah, \textit{Certain Properties of the Nielsen's $\beta$-Function}, Bull. Int. Math. Virtual Inst., 9(2019), 263-269.

\bibitem{Nantomah-2020-EJMS}  K. Nantomah, \textit{Harmonic Mean Inequalities for Hyperbolic Functions}, Earthline J. Math. Sci., 6(1)(2021), 117-129.

\bibitem{Nantomah-2021-IJAM}  K. Nantomah, \textit{A Harmonic Mean Inequality for the Exponential Integral Function}, Int. J. Appl. Math., 34(4)(2021), 647-652.

\bibitem{Nantomah-2021-AMSJ} K. Nantomah, \textit{A Harmonic Mean Inequality Concerning the Generalized Exponential Integral Function}, Adv. Math. Sci. J., 10(9)(2021), 3227-3231.

\bibitem{Nantomah-2023-AJMS} K. Nantomah, \textit{Degenerate Exponential Integral Function and its Properties}, Arab J. Math. Sci., Published Ahead-of-Print: \url{ https://doi.org/10.1108/AJMS-09-2021-0230}.


\bibitem{Niculescu-2000-MIA} C. P. Niculescu, \textit{Convexity according to the geometric mean}, Math. Inequal. Appl., 2(2)(2000), 155-167.

\bibitem{Pinelis-2002-JIPAM} I. Pinelis, \textit{L'hospital type Rules for Monotonicity, with Applications}, J. Inequal. Pure Appl. Math., 3(1)(2002), Art No. 5, 5 pages.

\bibitem{Yildirim-2020-JMI}  E. Yildirim, \textit{Monotonicity Properties on $k$-Digamma Function and its Related Inequalities}, J. Math. Inequal., 14(1)(2020), 161-173.

\bibitem{Yin-Etal-2018-ADE}  L. Yin, L-G. Huang, X-L. Lin and Y-L. Wang, \textit{Monotonicity, concavity, and inequalities related to the generalized digamma function}, Adv. Difference Equ., (2018) 2018:246.

\bibitem{Zhang-Chu-Zhang-2010-JIA} X-M. Zhang, Y-M. Chu and X-H. Zhang, \textit{The Hermite-Hadamard Type Inequality of GA-Convex Functions and Its Application}, J. Inequal. Appl., 2010(2010), Article ID 507560, 11 pages.


\end{thebibliography}


\end{document}